\documentclass{article}
\usepackage[utf8]{inputenc}
\usepackage{ amssymb }
\usepackage{ amsmath }
\usepackage{amsthm}
\usepackage{cleveref}
\usepackage{fancyhdr}
\usepackage{accents}
\newcommand\thickbar[1]{\accentset{\rule{.4em}{.8pt}}{#1}}
\fancypagestyle{footer}{\fancyhf{}\fancyfoot[L]{
L. J. Jolliffe, Department of Pure Mathematics and Mathematical Statistics, Centre for Mathematical Sciences, University of Cambridge, Wilberforce Road, Cambridge, CB3 0WB, United Kingdom\\
\emph{E-mail address:}  ljj33@cam.ac.uk}}
\newtheorem{theorem}{Theorem}

\newtheorem{definition}[theorem]{Definition}
\newtheorem{prop}[theorem]{Proposition}
\newtheorem{lemma}[theorem]{Lemma}
\newtheorem{cor}[theorem]{Corollary}

\crefalias{prop}{proposition}
\theoremstyle{plain}
\usepackage{titling}
\usepackage{xcolor}
\usepackage{tikz}
\usepackage{cleveref}
\newtheorem*{remark}{Remark}
\newtheorem*{theorem*}{Theorem}
\newcommand{\Fp}[0]{\mathbb{F}_p}
\newcommand{\val}[0]{\text{val}_p}
\newcommand{\modp}[0]{\hspace{2pt} ( \emph{mod } p )}
\title{Universal $p$-ary Designs}
\author{Liam Jolliffe}
\date{}
\begin{document}
\maketitle
\begin{abstract}
We investigate $p$-ary $t$-designs which are simultaneously designs for all $t$, which we call universal $p$-ary designs. Null universal designs are well understood due to Gordon James via the representation theory of the symmetric group. We study non-null designs and determine necessary and sufficient conditions on the coefficients for such a design to exist. This allows us to classify all universal designs, up to similarity.
\end{abstract}

\section{Introduction}

We shall briefly recall some definitions and facts from the theory of designs, more details can be found in \cite{Dembowski}. 
\begin{definition}
Let $[v]:=\{1,2,\dots,v\}$ be a finite set and $t\le s \le v$ be integers. An \emph{integral $t$-design on $[v]$ of constant block size} $s$ is a function $$u:[v]^s\to \mathbb{Z},$$ where $[v]^s$ is the set of all subsets of $[v]$ of size $s$, such that 
$$\hat{u}(Z):=\sum_{Y\supseteq Z}u(Y) = \mu_t \hspace{2em}\forall Z \in [v]^t $$
We call $\mu_t$ the \emph{coefficient} of the design and if $\mu_t=0$ then we say $u$ is a \emph{null-design}. We say that $\hat{u}$ is induced from $u$ and we shall denote its restriction to sets of size $j$ by $_j\hat{u}$.
\end{definition}
Similarly a $p$-ary $t$-design is a function $u:[v]^s\to \mathbb{F}_p$ such that $\hat{c}$ is constant on sets of size $t$. It is well known that integral $t$-design is also an integral $j$-design for all integers $0\le j\le t$, however this is not true over fields of positive characteristic. If a $p$-ary design, $u$, of constant block size, $s$, is a $j$ design for all $j < s$ then we say that $u$ is a \emph{universal} design for the partition $(v - s,s)$.  Graver and Jurkat \cite{GJ} determined when universal integral designs exist. 
\begin{theorem}\label{Existence of integral designs}\cite{GJ}
Let $t\le s \le v$ be integers. There exists a universal integral design for $(v-s,s)$ with coefficients $\mu_j$ if and only if $\mu_{j+1}=\frac{s-j}{v-j}\mu_j$ for $0\le j<t$.
\end{theorem}
The goal of this paper is to prove the equivalent result for $p$-ary designs and to describe the resulting universal designs, up to similarity. Null universal designs are well understood via a James' kernel intersection theorem, which was proved in the context of the representation theory of the symmetric group, but is re-stated in the language of designs as \Cref{kernel}. The existence, or otherwise, of non-null universal designs depends on the partition. We shall conclude our introduction with a number of definitions required to state the main result, and also some facts on divisibility of binomial coefficients.  

\begin{definition}
Let $u$ and $v$ be universal $p$-ary designs for $(a,b)$ with coefficents $\mu_j$ and $\gamma_j$ respectively . We say $u$ and $v$ are similar if there is some $k \in \Fp$ such that $\mu_j = k \gamma_j$.
\end{definition}
We shall now state some well known results on the divisibility of binomial coefficients, as many of the results in the theory of $p$-ary designs involve determining whether certain binomial coefficients are $0 \modp$ or not.

Let $a=\sum_{i=0}^\alpha a_ip^i$ be the base $p$ expansion of $a$; that is $0\le a_i\le p-1$ and $a_\alpha\ne 0$. The $p$-adic valuation $\val(a)$ is the least $i$ such that $a_i$ is non-zero, we call $\alpha$ the $p$-adic length of $a$ and write $l_p(a)=\alpha$. 
\begin{definition}
Let $(a,b)$ be a two part partition, that is $a\ge b >0$. We call a partition James if $\val(a+1)>l_p(b)$, while if $b=p^\beta + \hat{b}$ and $\hat{b} <p^{\val(a+1)} < p^\beta$ we call $(a,b)$ pointed.
\end{definition}

\begin{lemma}\label{Kummer}
Let $p$ be a prime and $a,b\in \mathbb{N}$, then $\val({a+b\choose b})$, the highest power of $p$ that divides ${a+b\choose b}$, is the number of carries that occurs when $a$ and $b$ are added in their base $p$ expansions. 
\end{lemma}
\begin{lemma}\label{binomial product}
Let $a=\sum_{i=0}^r a_ip^i$ and $b=\sum_{i=0}^r b_ip^i$, with $0\le a_i,b_i\le p-1$. Then 
$${a\choose b}\equiv{a_0 \choose b_0}{a_1 \choose b_1}\cdots{a_r \choose b_r} \modp.$$
In particular, ${a\choose b}\equiv 0 \modp$ if and only if some $a_i<b_i$.
\end{lemma}
\begin{lemma}\label{binomial}
Let $a,b\in \mathbb{N}$. The binomial coefficients ${a+1 \choose 1},{a+2 \choose 2},\dots{a+b \choose b}$ are all divisible by $p$ if and only if $a\equiv -1 \hspace{2pt} ( \emph{mod } p^{l_p(b)} )$.
\end{lemma}
\begin{remark}
This gives an alternative characterisation of a James partition, in particular $\lambda=(a,b)$ is James if and only if $a \equiv -1 \hspace{2pt} ( \emph{mod } p^{l_p(b)} )$, or equivalently $l_p(b) < \val(a+1)$ for all $i<r$.
\end{remark}
We can now state our main result:
\begin{theorem}\label{combinatorial main theorem}
Let $a,b\in \mathbb{N}$, with $a\ge b$ and let $u$ be a non-null universal $p$-ary design for $(a,b)$. If $(a,b)$ is neither pointed or James, then $u$ is similar to the constant design. If $(a,b)$ is James then $u$ is unique up to similarity, while if $(a,p^\beta+\hat{b})$ is pointed then $u=u'+c$ where $u'$ is non-null only as a $\hat{b}$-design, while $c$ is similar to the constant design. 
\end{theorem}

\section{Uniqueness of $p$-ary Designs}

Universal null $p$-ary designs are well understood, due to the work of James on the representation theory of the symmetric group. James' well-known kernel intersection theorem gives a characterisation of the Specht module $S^{(a,b)}$ as the collection of all null universal $p$-ary designs for $(a,b)$ \cite{James 77}. 

\begin{theorem}\label{kernel}

Let $X,Y\subseteq [v]$ with $|X| = |Y|=b$ and $X\cap Y=\emptyset $. Let $f:X\to Y$ be a bijection. Define 
$$
u:[v]^b\to \Fp
$$
by, 
$$u(Z)=\begin{cases} (-1)^{|Z\cap Y|} &\text{if } (Z\subseteq X\cup Y) \wedge(\forall x\in X)( x\in Z \implies f(x)\not\in Z) \\
            0 &\text{otherwise }.\end{cases} $$
Then $u$ is a null $p$-ary design for $(v-b,b)$. Moreover any null $p$-ary design for $(v-b,b)$ is a linear combination of designs of this form.
\end{theorem}

Non-null designs also play an important role in the representation theory of the symmetric group, as they determine certain non-split extensions of Specht modules, investigated by the author in \cite{Jolliffe}. \Cref{Existence of integral designs} describes the relationship between the coefficients of integral designs, and a similar analysis determines when a $p$-ary $t$-design is also a $j$-design.

The \textit{inclusion matrix,} $A_{i}^b(v)$, where $i\le b \le v$, is  the ${v \choose i} \times {v \choose b}$ matrix whose rows are indexed by subsets of $[v]:=\{1,2,\dots, v\}$ of size $i$ and whose columns are indexed by subsets of $[v]$ of size $b$. The entry corresponding to position $X,Y$ is $1$ if $X\subseteq Y$ and $0$ otherwise. Gottlieb showed matrix is of full rank over fields of characteristic 0 \cite{Gottlieb}. If $u$ is an integral design of block size for $(v-b,b)$, then considering $u$ as a vector of length ${v\choose b}$, we see that 
$$
A_{i}^b(v)u=\mu_i \mathbf{1_i},
$$
where $\mathbf{1_i}$ is the vector of length ${v \choose i}$ consisting of 1's. It is clear that 
$$
A_{j}^i(v)A_{i}^b(v)={b-j\choose i-j} A_{j}^b(v),
$$
and thus 
$$
{v-i\choose i-j}\mu_i={b-j\choose i-j}\mu_j,
$$
proving the necessity of the conditions in \Cref{Existence of integral designs}.

\begin{prop}\label{relationship between coefficients}
Let $u:[v]^b\to k$ be a $p$-ary $t$-design of block size $b$ on a set of size $v$ with coefficient $\mu_t$. Let $j\le b$ be such that ${b-j\choose t-j}\not\equiv 0 \modp$, then $u$ is also a $j$-design, with coefficient 
$$
\mu_j=\frac{{v-j\choose t-j}}{{b-j\choose t-j}}\mu_t.
$$
\end{prop}
\begin{proof}
\begin{align*}
{b-j\choose t-j} A_{j}^b(v)u&=A_{j}^t(v)A_{t}^b(v)u  \\
                            &=A_{j}^t(v)\mu_t\mathbf{1_t}\\
                            &={v-j\choose t-j}\mu_t\mathbf{1_t}.
\end{align*}
\end{proof}
\begin{remark}
Wilson \cite{Wilson} showed that there are examples of $t$-designs which are not $j$ designs whenever ${b-j\choose t-j}\equiv 0 \modp$, which is very different to the behaviour of integral designs.
\end{remark}
In light of this result, to check a design is universal it suffices to check that it is a $b-p^l$-design for all $l\le l_p(b)$. 
\begin{prop}
Let $\lambda=(a,b)$. A design for $\lambda$ is universal if and only if it is a $(b-p^l)$-design for all $l\le l_p(b)$.
\end{prop}
\begin{proof}
Of course a universal design is a $(b-p^l)$-design. A $(b-p^l)$-design, is also a $j$ design for all $j<b-p^l$ with ${b-j \choose b-p^l-j}\ne 0$; that is, for any $j$ such that the sum $(b-j-p^l)+p^l$ has no carries in $p$-ary notation, by \Cref{Kummer}. This is precisely those $j$ for which the coefficient of $p^l$ in the $p$-ary expansion of $b-j$, which we shall denote $(b-j)_l$, is non zero. If $j<b$, then some $(b-j)_l\ne 0$, and as $u$ is a $(b-p^l)$-design $u$ is also a $j$-design by \Cref{relationship between coefficients}.
\end{proof}
Wilson has determined when non-null $p$-ary $t$-designs exist. 
\begin{theorem}\label{Wilson}\cite{Wilson}
Let $t\le b\le v-t$. Then there is a non-null $p$-ary $t$-design of block size $b$ if and only if 
$$
{b-i\choose t-i}\equiv 0 \modp \hspace{12pt} \text{implies}\hspace{12pt} {v-i\choose t-i}\equiv 0 \modp
$$
for all $i\le t$.
\end{theorem}

\begin{cor}\label{non-null designs}
There are non-null $p$-ary $(b-p^l)$-designs for $(a,b)$ if and only if $a_l\not\equiv -1 \modp$ or $b\le p^{l+1}$.
\end{cor}
\begin{proof}
By \Cref{Wilson} a non-null $(b-p^l)$-design exists if  
$$
{b-j\choose p^l}\equiv 0 \modp \hspace{12pt} \text{implies}\hspace{12pt} {a+b-j\choose a+p^l}\equiv 0 \modp.
$$
If $a_l\equiv -1 \modp$ and $b>p^{l+1}$ then setting $j=b-p^l$ we see that non-null designs can not exist. On the other hand if $b\le p^{l+1}$ then ${b-j\choose p^l}\not\equiv 0 \modp $ for all $j<b-p^l$ so there are non-null $(b-p^l)$-designs. Finally, if $a\not\equiv -1 \modp$ then  ${b-j\choose p^l}\equiv 0 \modp $ whenever $(b-j)_l=0$. If $(b-j)_l=0$ then the sum $(a+p^l)+(b-j-p^l)$ necessarily has a carry in $p$-ary notation, so ${a+b-j\choose a+p^l}\equiv 0 \modp$ by \Cref{Kummer}.
\end{proof}
Combining this with the relationship between coefficients, established in \Cref{relationship between coefficients}, we obtain more integers $j$ for which a universal design for $(a,b)$ is null. 
\begin{prop}\label{non-null designs 2}
If a universal design, $u$, for $(a,b)$ is non-null as a $j$-design, then $(b-j)_m+a_m< p$ for all $m< l_p(b)$.
\end{prop}
\begin{proof}
Suppose $u$ is non-null as a $j$-design with coefficient $\mu_j$, and let $m<l_p(b)$ be such that $(b-j)_m\ne 0$. As $u$ is non-null for $j$, we must have $u$ is non-null for $b-p^m$, by \Cref{relationship between coefficients}, as $$\mu_j=\frac{{a+b-j\choose b-p^m-j}}{{b-j\choose b-p^m-j}}\mu_{b-p^m}.$$
For $u$ to be non-null as a $j$-design, we must have ${a+b-j\choose b-p^m-j}\ne 0$. 
\Cref{non-null designs} ensures that $a_m\not\equiv -1 \modp $ and thus $(a+p^m)+(b-j-p^m)$ having no carries is equivalent to $(a)+(b-j)$ having no carries. Using \Cref{Kummer} we see that if $u$ is non-null then $(a)+(b-j)$ has no carries, and therefore $(b-j)_m+a_m< p$ for all $m< l_p(b)$.
\end{proof}
Our next goal is to determine what the relationship is between the non-zero coefficients of a universal design. Let $u$ be a universal design for $(a,b)$, and let $X$ be the set of all $j$ with $(b-j)_m+a_m< p$ for all $m< l_p(b)$. Observe if $j \notin X$ then $u$ must be a null $j$-design, and so $X$ contains all $j$ such that $u$ is a non-null $j$-design. We shall define a partial ordering on $X$ by setting $i\ge_X j$ if $i>j$ and ${b-j \choose i-j}\not\equiv 0\modp$. If $i\ge_X j$ and $\mu_i$ and $\mu_j$ are the coefficients of $u$ corresponding to $i$ and $j$ respectively, then  $\mu_j=\frac{{a+b-i\choose j-i}}{{b-i\choose j-i}}\mu_j$, so we have a relationship between the coefficients appearing in the same connected component of $X$.
\begin{prop}\label{connected components 1}
If $\lambda=(a,b)$ is James, then $X$ has a single connected component.
\end{prop}
\begin{proof}
Recall if $\lambda$ is James then $b<p^{\val(a+1)}$, and $a_m\equiv -1 \modp$ for all $m<l_p(b)$. Write $b=\alpha p^\beta + \hat{b}$, where $\beta=l_p(b)$ and $\hat{b}<p^\beta$, and observe, by \Cref{non-null designs 2}, that $X=\{\hat{b},p^\beta+\hat{b},\dots,(\alpha-1)p^\beta+\hat{b}\}$, which, by \Cref{binomial product}, is a single connected component.
\end{proof}
\begin{prop}\label{connected components 2}
If $\lambda=(a,b)$ is not James, and $b= \alpha p^\beta + \hat{b}$ then $X$ has a single connected component, unless $\lambda$ is pointed, in which case $X$ has two connected components, one of which consists only of the element $\hat{b}$.
\end{prop}
\begin{proof}
Observe that $i,j\in X$ are comparable if and only if $(b-i)_m\le (b-j)_m$ for all $m\le l_p(b)$, or $(b-j)_m\le (b-i)_m$ for all $m\le l_p(b)$. Observe also that $(b-i)_m=0$ for all $m<l_p(b)$ for which $a_m\equiv -1 \mod p$. The join of $i,j\in X$, if it exists, is the element $i\vee j = x$ such that $(b-x)_m=\text{max}\{(b-i)_m,(b-j)_m\}$, the meet, $y=i\wedge j$, is the element $y$ such that $(b-y)_m=\text{min}\{(b-i)_m,(b-j)_m\}$. These may fail to be in $X$ as it may be that $(b-x)>b$ or $b-y=0$. If, however, $(b-i)_m$ and $(b-j)_m$ are both non-zero for some $m$ then $i\wedge j \in X$. 

Let $x$ be such that $(b-x)_m=p-1-a_m$ for $m<\beta$ and $(b-x)_\beta=\alpha-1$, and observe that $x\in X$ by \Cref{non-null designs 2}. Clearly $j\in X$ with $j>\hat{b}$ is comparable to $x$. If $j < \hat{b}\in X$, or if $j=\hat{b}$ and $\alpha \ne 1$ then $x\wedge j \in X$. 

It only remains to consider the case where $j=\hat{b}$ and $\alpha = 1$, which, if $\hat{b}>p^{\val(a+1)}$ is clearly comparable to $\hat{b}-p^{\val(a+1)}$, which is in the same component as $x$. It follows that if $\lambda$ is not pointed then there is only one connected component of $X$. 

On the other hand, when $\lambda$ is pointed $\hat{b}$ is not comparable to any other element and thus is in a connected component of its own. This is as no $j<\hat{b}$ is in $X$ as no $j<\hat{b}$ has $(b-j)_m=0$ for all $m<l_p(b)$ where $a_m\equiv -1 \mod p$. Similarly no $j>\hat{b}$ has $(b-j)_\beta \ge 1$, so $j$ and $\hat{b}$ are incomparable .
\end{proof}
\begin{proof}[Proof of uniqueness in \Cref{combinatorial main theorem}]
If $u$ is a universal design for $(a,b)$, then its coefficients are entirely determined by the connected components of $X$, thus an understanding of this poset allows us to determine the possible coefficients of designs. If $(a,b)$ is not pointed, then non-null universal designs, if they exist, are unique up to similarity, while if $(a,b)$ is pointed, then any design must be the sum of two designs, uniquely determined by its coefficients on each of the two connected components of $X$.
\end{proof}
\section{Existence of designs}

In the previous section we have seen a complete characterisation of universal null $p$-ary designs and described, up to similarity, the uniqueness of non-null universal designs. We now move to considering the existence of non-null designs for $(a,b)$. We first consider $p$-ary designs which come from the mod $p$ reduction of integral designs. Clearly the constant design, $c_{(a,b)}$, is an integral design, with coefficients $\mu_i={a+b-i\choose b-i}$, and therefore is null if and only if $(a,b)$ is James.
\begin{prop}
Let $(a,b)$ be a partition which is neither James nor pointed. Then the constant design is the unique, up to similarity, universal $p$-ary design for $(a,b)$.
\end{prop}
\begin{prop}\label{James}
Let $\lambda=(a,b)$, then there exists an integral design which is not similar to the constant design if and only if $\lambda$ is James.
\end{prop}
\begin{proof}
Any integral design must have coefficients satisfying the conditions of \Cref{Existence of integral designs}, $\mu_{s+1}=\frac{b-s}{a+b-s}\mu_s$ for $0\le s<t$. This means that
$$
\mu_s=\frac{{a+b-s\choose a}}{{a+b\choose a}}\mu_0.
$$
To ensure that some $\mu_i\not\equiv 0 \modp$ we must take $\mu_s=k\frac{{a+b-s\choose a}}{p^d}$ where $k\in \Fp$ is non-zero and $d$ is the least power of $p$ dividing some ${a+b-s\choose a}$ for $s\in \{0,1,\dots,b-1\}$. That is, $d=\text{min}_{s<b}\{\val{a+b-s\choose b}\}$. Observe that 
\begin{align*}
    _j\widehat{c_{(a,b)}}&={a+b-j \choose b-j}\cdot c_{(a+b-j,j)}\\
&=k^{-1}p^d\mu_j \cdot c_{(a+b-j,j)},
\end{align*}
 and so if $p^d$ is a unit in $\Fp$, that is if $d=0$, then $$\psi_{1,j}(k\cdot c_{(a,b)}-u)=0,$$ and $u$ is not similar to the constant design. This means $u$ is similar to the constant design if and only if $p\mid {a+b-j\choose a}$ for all $j \in \{0,1,\dots,b-1\}$, which by \Cref{binomial} is if and only if $\lambda$ is James. 
\end{proof}
\begin{theorem}
The unique, up to similarity, universal $p$-ary design for a James partition $(a,b)$ is the mod $p$ reduction of the integral design with coefficients $\mu_s=\frac{{a+b-s\choose a}}{p^d}$ where $d=\text{min}_{s<b}\{\val{a+b-s\choose b}\}$.
\end{theorem}
We have seen that if $(a,b)$ is pointed then the constant design is non-null. We shall now construct another non-null design for $(a,b)$ which is not similar to the constant design, completing the classification.  
\begin{prop}
Let $(a,b)$ be such that $b=p^\beta$ and $\val(a+1)<\beta$. Then there exists a universal $p$-ary design which is null as a $t$-design for all $t>0$ and non-null as a $0$-design. 
\end{prop}
\begin{proof}
Let $m=a-b+1$ and define $$u(Y)=\begin{cases}
1 & \text{if } Y\cap [m] = \emptyset \\
0 & \text{otherwise}
\end{cases}$$ for $|Y|=b$.
Then, for $|Z|=s$ $$\hat{u}(Z) = \begin{cases} {a+b-m-s \choose b-s} & \text{if } Z\cap [m]  = \emptyset \\
0 & \text{otherwise}.
\end{cases}
$$ 
By our choice of $m$ the coefficients are ${a-m+1 \choose 1},{a-m+2 \choose 2},\dots,{a-m+b-1 \choose b-1}$ are all divisible by $p$, by \Cref{binomial}. Then $u$ is a universal design which is non-null only as a $0$-design.
\end{proof}
Let $u$ be the design constructed above for the partition $(a,p^\beta)$. We shall modify $u$ to construct a design for a pointed partition $(a,p^\beta+\hat{b})$, where $\hat{b}<p^{\val(a+1)}<p^\beta$, which is non-zero only as a $\hat{b}$-design.

Let $u:[v]^{p^\beta}\to \Fp$ be the design constructed above for the partition $(a,p^\beta)\vdash v$. We shall modify $u$ to construct a design for a pointed partition $(a,p^\beta+\hat{b})$, where $\hat{b}<p^{\val(a+1)}<p^\beta$, which is non-zero only as a $\hat{b}$-design. Let $Y=\{a+p^\beta+1,\dots,a+b\}$, then $Y$ is a set of size $\hat{b}$. 

Define $u_Y:[v]^{p^\beta+\hat{b}}\to \Fp$ by 
$$u_Y(Z)
=\begin{cases} 
u(Z\backslash Y)  & \text{ if } Y\subseteq Z \\
0 & \text{ otherwise },  
\end{cases}$$
and $u^Y:[v]^{p^\beta}\to \Fp$ by
$$u^Y(Z)
=\begin{cases} 
u(Z)  & \text{ if } Y\cap Z = \emptyset\\
0 & \text{ otherwise. }
\end{cases}$$
Given a subset $X\subset [v]$ we denote by $\delta_X:[v]^{|X|}$ the indicator function; that is
$$
\delta_X(Y) = \begin{cases}
1 &\text{ if } X=Y\\
0 &\text{ otherwise.}
\end{cases}
$$
Of course, these functions may not be designs, but we may consider the functions they induce on subsets of $[v]$ as before. 
Consider $_j\widehat{u_Y}:[v]^j\to \Fp$, by grouping terms by the size of their intersection with $Y$. First, consider the case where $\hat{b}<j<b$:
$$
\widehat{u_Y}=(_{j-\hat{b}}\hat{u})_Y+\sum_{y\in Y} (_{j-\hat{b}-1}\hat{u})_{Y\backslash \{y\}}^y+\cdots+(_j\hat{u})^Y.
$$
Each of these terms is $0$, by our choice of $u$, so $_j\hat{u_Y}$ is 0.

Similarly for $j\le \hat{b}$
\begin{align*}
_j\widehat{u_Y}&=\sum_{i=0}^j \sum_{\mid Y'\cap Y\mid=i}\hat{_{j-i}(u)}_{Y'}\\
&=\sum_{\mid Y'\cap Y\mid=j}\hat{(_0u)}_{Y'}\\
&=\mu_0 \cdot _j\delta_Y,
\end{align*} 
where $\mu_0\ne 0$ is the coefficient of $u$ as a 0-design.

Observe that if $Y$ is any subset of $[a+b]$, not necessarily $\{a+p^\beta+1,\dots,a+b\}$, then we may define $u_Y$ as before, by first defining $u$ on subsets of $[a+b]\backslash Y$ of size $p^\beta$.

Let $X\subseteq [a+b]$ of size $b-1=p^{\beta}+\hat{b}-1$. Define $u_{\thickbar{X}}:=\sum_{Y\subseteq X}u_Y$. Then 
$$
\widehat{u_{\thickbar{X}}}=\sum_{Y\subseteq X}\widehat{u_Y},
$$
which is $0$ when restricted to sets of size $j$ if  $\hat{b}<j<b$. When $j\le \hat{b}$,
\begin{align*}
_j \widehat{u_{\thickbar{X}}}&=\sum_{Y\subseteq X} {}_j \widehat{u_Y}\\
                     &=\sum_{Y\subseteq X} \mu_0 \cdot _j\delta_Y\\
                     &={p^\beta-1+\hat{b}-j \choose \hat{b}-j}\mu_0 \cdot _j\delta_X,\\
\end{align*}
which is 0 if $j\ne \hat{b}$. So
$$
_{\hat{b}}\widehat{u_{\thickbar{X}}}=\mu_0 \cdot _{\hat{b}}\delta_X.
$$
If $\mathcal{ U}$ is a non null $p$-ary $\hat{b}$-design of block size $b-1$ and coefficient $\alpha$\, then setting
$$
u_\mathcal{U}:=\sum_{X}\mathcal{U}
(X)u_{\thickbar{X}},
$$
where the sum is over all sets $X$ of size $b-1$ and $\mathcal{U}(X)$ is the coefficient of $X$ in the $\hat{b}$-design $\mathcal{U}$, we see 
\begin{align*}
_{\hat{b}}\widehat{u_\mathcal{U}}&=\sum_{X}\mathcal{U}(X)\widehat{u_{\thickbar{X}}}\\
              &=\sum_{X}\mathcal{U}(X)\mu_0 \cdot _{\hat{b}}\delta_X\\
              &=\alpha \mu_0 \cdot _{\hat{b}}\delta_X,
\end{align*}
and of course 
$$
_j\widehat{u_\mathcal{U}}=0
$$
for all $j\ne \hat{b}$.
\begin{theorem}\label{pointed}
Let $\lambda=(a,b)$ be such that $b=p^\beta+\hat{b}$ and $\hat{b}<p^{\val(a+1)}<b$. Then there is a universal design which is non-null only as a $\hat{b}$-design.
\end{theorem}
\begin{proof}
An element of the form $u_\mathcal{U}$ as described above is such a design, it remains to prove such an element exists; that is that there is a non null $p$-ary $\hat{b}$-design of block size $b-1$. By \Cref{Wilson}, we may construct such a design if (and only if) ${a+b-1-i\choose \hat{b}-i}\equiv 0 \modp$ whenever ${b-1-i\choose \hat{b}-i}\equiv 0 \modp $. Of course  ${b-1-i\choose \hat{b}-i}={p^\beta +\hat{b}-1-i\choose \hat{b}-i} \equiv 0 \modp $ for all $i<\hat{b}$, so it remains to see that ${a+b-1-i\choose \hat{b}-i}\equiv 0 \modp$ for all $i<\hat{b}$; that is, that ${a+p^\beta+j\choose j}\equiv 0 \modp$ for all $j<\hat{b}$. This follows from \Cref{binomial}, as $a+p^\beta \equiv -1 \hspace{2pt} ( \emph{mod } p^{l_p(\hat{b})} )$. 
\end{proof}

\begin{proof}[Existence of $p$-ary designs]
If $(a,b)$ is James, then the construction of Graver and Jurkat \cite{GJ} gives rise to a non-null design. If $(a,b)$ is not James then the constant design is non-null. If $(a,b)$ is pointed then \Cref{pointed} gives a non-null universal design. Completing the proof of \Cref{combinatorial main theorem}, which we state again below to conclude.
\end{proof}
\begin{theorem*}
Let $a,b\in \mathbb{N}$, with $a\ge b$ and let $u$ be a non-null universal $p$-ary design for $(a,b)$. If $(a,b)$ is neither pointed or James, then $u$ is similar to the constant design. If $(a,b)$ is James then $u$ is unique up to similarity, while if $(a,p^\beta+\hat{b})$ is pointed then $u=u'+c$ where $u'$ is non-null only as a $\hat{b}$-design, while $c$ is similar to the constant design. 
\end{theorem*}
\section*{Acknowledgements}
This work will appear in the author's PhD thesis prepared at the University of Cambridge and supported by the Woolf Fisher Trust and the Cambridge Trust. This work was done while the author was a visiting scholar at Victoria University of Wellington. The author would like to thank his supervisor Dr Stuart Martin for his encouragement and support.

\thispagestyle{footer}
\end{document}